\documentclass[runningheads]{llncs}
\usepackage{graphicx}
\usepackage{amsmath}
\usepackage{hyperref}
\usepackage{xcolor}

\begin{document}
\newcommand{\revision}[1]{{\color{purple} #1}}
\newcommand{\mynote}[1]{{\color{blue} #1}}
\title{Circular Economy Design through System Dynamics Modeling
}

\author{Federico Zocco\orcidID{0000-0002-6631-7081} \and
Monica Malvezzi\orcidID{0000-0002-2158-5920}}
\authorrunning{F. Zocco and M. Malvezzi}
\institute{Department of Information Engineering and Mathematics, University of Siena, Italy\\
\email{federico.zocco.fz@gmail.com}\\
\email{monica.malvezzi@unisi.it}
}
\maketitle              %
\begin{abstract}
Nowadays, there is an increasing concern about the unsustainability of the take-make-dispose paradigm upon which traditional production and consumption systems are built. 
The concept of circular economy is gaining attention as a potential solution, but it is an emerging field still lacking analytical and methodological dynamics approaches.
Hence, in this paper, firstly we 
propose a quantitative definition of
circularity, namely, $\lambda$, predicated on compartmental dynamical thermodynamics, and then, we use it to state the optimization of the circularity $\lambda$ as an arg-max problem.  By leveraging the derivation of Lagrange's equations of motion from the first law of thermodynamics, we 
apply the   
analytical
mechanics approaches to
 circularity. Three examples illustrate the calculation of $\lambda$ for different settings of two compartmental networks. 
In particular, hypothesizing a repair stage followed by product reuse we highlight the memory property of $\lambda$. Finally, robotic repair is proposed within this framework to pave the way for circular robotics as a new area of research in which the performance of a robotic system is measured against $\lambda$.   
\end{abstract}

\section{Introduction}
Raw materials and minerals are the constituents of any product, from the tiny transistor to the large football stadium. With the traditional (i.e., linear) economy, there are open issues regarding the criticality of several raw materials \cite{CRM-EU} and regarding the management of end-of-life products \cite{EU-waste}. The concept of circularity at the core of a circular economy (CE) is gaining importance to address both material supply uncertainties and waste generation \cite{EllenMacArthurFound}. The practical strategies to increase circularity at local and global scales are the so-called ``Rs'' such as reduce, reuse, repair, refurbish, and re-manufacture \cite{potting2017circular}. 
As a research field mainly rooted in business and environmental programs, CE currently lacks 
model-based, dynamics foundations.
Hence, this paper contributes to exploiting system dynamics in circularity to provide mathematical methods and quantitative evaluations that can strengthen and clarify the knowledge of CE.

\textbf{Related Work:}      
This brief paper builds on the systemic framework recently proposed by the authors \cite{zocco2024unification}. The framework is complementary to the state-of-the-art methodology of material flow analysis (MFA) \cite{cullen2022material,wang2024bayesian} because leveraging ordinary differential equations, which underpin dynamical systems and control theories, rather than data analysis. In addition, this paper provides a measure of circularity similar to previous works \cite{basile2024measuring}, but to the best of our knowledge, it is the first to do so through the lenses of analytical mechanics.

\section{Preliminaries}
Consider the modeling approach of the Rankine cycle, which consists of applying the mass balance and the thermodynamic laws to each component of the cycle one at a time \cite{moran2010fundamentals}. Each component of the cycle is contained inside a \emph{control surface}. Hence, each component of the cycle is a \emph{thermodynamic compartment} contained inside a control surface. Given the generality of the mass balance and the laws of thermodynamics \cite{haddad2017thermodynamics,haddad2019dynamical}, the modeling approach of the Rankine cycle can be generalized as follows.

Let $c^k_{i,j}$ be the $k$-th thermodynamic compartment through which the material moves from compartment $i$ to compartment $j$. Thus, the Rankine cycle can be depicted as in Fig. \ref{fig:RankineCycle}, where the control surfaces are indicated with the dashed lines on the left-hand side and with the rectangles on the right-hand side, and where $\beta$ is the working fluid. The position of an infinitesimal mass element $\beta$ is tracked by a geometrical vector with respect to the origin $O$ of a reference frame. 
\begin{figure}
    \vspace{-0.5cm}
    \centering
    \includegraphics[width=0.75\textwidth]{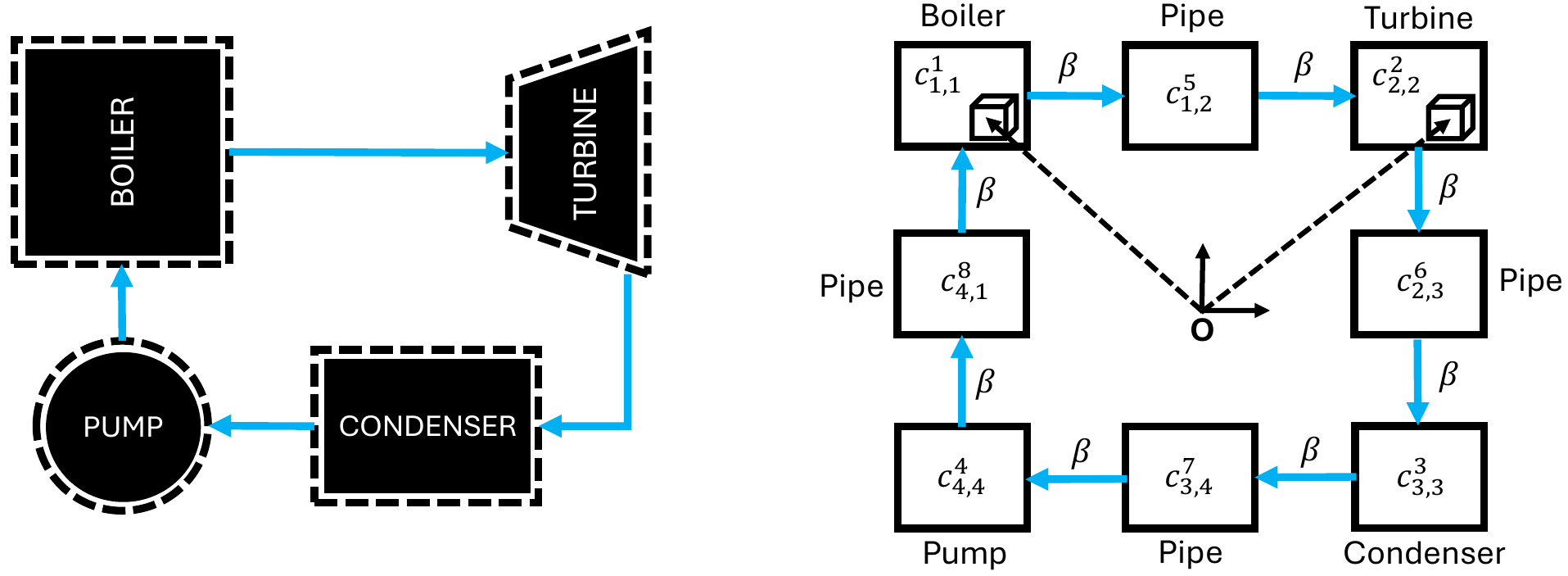}
    \caption{Classic (left) and generalized (right) representations of the Rankine cycle. Figure adapted from \cite{zocco2024unification}.}
    \label{fig:RankineCycle}
    \vspace{-0.5cm}
\end{figure}
Note now that the compartmental representation on the right-hand side of Fig. \ref{fig:RankineCycle} is a network of compartments. Hence, it can be compactly represented by a graph, detailed in the following.
\begin{definition}[\hspace{1sp}\cite{bondy1976graph}]
A directed graph $D$ or \emph{digraph} is a graph identified by a set of $n_v$ \emph{nodes} $\{v_1, v_2, \dots, v_{n_v}\}$ and a set of $n_a$ \emph{arcs} $\{a_1, a_2, \dots, a_{n_a}\}$ that connect the nodes. A digraph $D$ in which each node or arc is associated with a \emph{weight} is a \emph{weighted digraph}. 
\end{definition}
Using the compartmental perspective and the notion of digraph, we now introduce the definition of thermodynamical material network (TMN).
\begin{definition}[\cite{zocco2023thermodynamical}]\label{def:TMN}
A \emph{thermodynamical material network} (TMN) is a set $\mathcal{N}$ of connected thermodynamic compartments, that is, 
\begin{equation}\label{def:TMNset}
\begin{gathered}
\mathcal{N} = \left\{c^1_{1,1}, \dots, c^{k_v}_{k_v,k_v}, \dots, c^{n_v}_{n_v,n_v}, \right. \\ 
\left. c^{n_v+1}_{i_{n_v+1},j_{n_v+1}}, \dots, c^{n_v+k_a}_{i_{n_v+k_a},j_{n_v+k_a}}, \dots, c^{n_c}_{i_{n_c},j_{n_c}}\right\}, 
\end{gathered}
\end{equation}
which transport, store, use, and transform a target material. Each compartment is indicated by a \emph{control surface} and it is modeled using \emph{dynamical systems} derived from a mass balance and/or at least one of the laws of \emph{thermodynamics} \cite{haddad2019dynamical}.
\end{definition}

Specifically, $\mathcal{N} = \mathcal{R} \cup \mathcal{T}$, where $\mathcal{R} \subseteq \mathcal{N}$ is the subset of compartments $c^k_{i,j}$ that \emph{store}, \emph{transform}, or \emph{use} the target material, while $\mathcal{T} \subset \mathcal{N}$ is the subset of compartments $c^k_{i,j}$ that \emph{move} the target material between the compartments belonging to $\mathcal{R} \subseteq \mathcal{N}$. A net $\mathcal{N}$ is associated with its weighted \emph{mass-flow digraph} $M(\mathcal{N})$, which is a weighted digraph whose nodes are the compartments $c^k_{i,j} \in \mathcal{R}$ and whose arcs are the compartments $c^k_{i,j} \in \mathcal{T}$. For node-compartments $c^k_{i,j} \in \mathcal{R}$ it holds that $i = j = k$, whereas for arc-compartments $c^k_{i,j} \in \mathcal{T}$ it holds that $i \neq j$ because an arc moves the material from the node-compartment $c^i_{i,i}$ to the node-compartment $c^j_{j,j}$. The orientation of an arc is given by the direction of the material flow. The superscript $k$ is the identifier of each compartment. The weight assigned to a node-compartment is the mass stock $m_k$ within the corresponding compartment, whereas the weight assigned to an arc-compartment is the mass flow rate $\dot{m}_{i,j}$ from the node-compartment $c^i_{i,i}$ to the node-compartment $c^j_{j,j}$. The superscripts $k_v$ and $k_a$ in (\ref{def:TMNset}) are the $k$-th node and the $k$-th arc, respectively, while $n_c$ and $n_v$ are the total number of compartments and nodes, respectively. Since $n_a$ is the total number of arcs, it holds that $n_c = n_v + n_a$ \cite{zocco2023thermodynamical,zocco2022circularity}.

So far, we first recalled the well-known Rankine cycle, and then, we generalized its modeling approach as a sequence of thermodynamic compartments $\mathcal{N}$ (\ref{def:TMNset}) connected through the flow of a material $\beta$ of interest. Subsequently, we introduced the \emph{mass-flow digraph} of $\mathcal{N}$, which provides a compact representation of $\mathcal{N}$.

Now, since our goal is to design \emph{circular} networks $\mathcal{N}$, we need to define a measure of circularity. We will do so after the following definition.
\begin{definition}[Unsustainable mass or flow]\label{def:unsMassFlow}
A mass or flow is \emph{unsustainable} if either it exits a nonrenewable reservoir or it enters a landfill, an incinerator, or the natural environment as a pollutant.
\end{definition}
The locations mentioned in Definition \ref{def:unsMassFlow}, i.e., reservoirs, landfills, incinerators, and the environment, are thermodynamic compartments $c^k_{i,j} \in \mathcal{N}$. We can now define the circularity of $\mathcal{N}$.  
\begin{definition}[Circularity]
The circularity $\lambda(\mathcal{N}) \in (-\infty, 0]$ is defined as
\begin{equation}\label{eq:circularity}
\lambda(\mathcal{N}) = - \left(m_{\textup{u},\textup{b}} + \dot{m}_{\textup{u},\textup{c}}\Delta\right).
\end{equation}
\end{definition}
In (\ref{eq:circularity}), $m_{\textup{u},\textup{b}}$ is the total unsustainable mass transported in batches (e.g., solids transported on trucks), $\dot{m}_{\textup{u},\textup{c}}$ is the total unsustainable flow transported continuously (e.g., fluids transported through pipes), and $\Delta > 0$ is a constant interval of time introduced as a multiplying factor in order to convert the flow $\dot{m}_{\textup{u},\textup{c}}$ into a mass, and hence, to make the sum with $m_{\textup{u},\textup{b}}$ physically consistent. The choice of $\Delta$ is arbitrary, but its value must be kept the same for any calculation of $\lambda(\mathcal{N})$ in order to make comparisons. For simplicity and without loss of generality, we assume $\Delta = 1$ s.    

Thus, the design of circular systems $\mathcal{N}$ can be stated as
\begin{equation}
\mathcal{N}^* = \text{arg }\text{max} \quad \lambda(\mathcal{N}),
\end{equation}
which consists of finding the optimal compartmental network $\mathcal{N}$ that maximizes the circularity (\ref{eq:circularity}) or, equivalently, that minimizes the total unsustainable mass and flow. In the next section, we will use analytical mechanics to model the dynamics of the material $\beta$ for a given $\mathcal{N}$. This will enable us to calculate the circularity $\lambda(\mathcal{N})$ (\ref{eq:circularity}), and then, to analyse how the circularity changes with different settings and configurations of $\mathcal{N}$. Before doing so, the following result must be stated. 

Let $T$ be the kinetic energy, let $V$ be the potential energy, and let $L = T - V$ be the Lagrangian function of an element of material $\beta$ and mass $m$. Let $s$ be the coordinate associated with the element path and let $\xi$ be the sum of the nonconservative forces acting on the element. Thus, the Lagrange's equation of motion is given by
\begin{equation}\label{eq:LagrangeEquation}
\frac{\text{d}}{\text{d}t}\left(\frac{\partial L}{\partial \dot{s}}\right) = \frac{\partial L}{\partial s} + \xi. 
\end{equation}
\begin{proposition}\label{prop:LagrFromFirstLaw}
The Lagrange's equation of motion (\ref{eq:LagrangeEquation}) can be derived from the first law of thermodynamics. 
\end{proposition}
\begin{proof}
The proof is given in \cite{zocco2023thermodynamical}, Equations (22)-(27).
\end{proof}
From Proposition \ref{prop:LagrFromFirstLaw} it follows that the application of Lagrange's equation (\ref{eq:LagrangeEquation}) to a mechanical system is an application of the first law of thermodynamics, and hence, a mechanical system can, in general, be represented as a thermodynamic compartment $c^k_{i,j} \in \mathcal{N}$.

\section{Analytical Mechanics of Circularity through Examples}
This section covers three examples. Specifically, it begins with a linear configuration of $\mathcal{N}$ (Example 1), and subsequently, it hypothesizes material use reduction (Example 2) and robotic repair (Example 3) to analyse their effect on the circularity $\lambda(\mathcal{N})$ defined in Eq. (\ref{eq:circularity}).

\subsection{Example 1: Linear Case}\label{sub:Example1}
Consider the black nodes and the black arcs in Fig. \ref{fig:Example1}, which give the following \emph{linear} system, namely, $\mathcal{N}_{\text{l}}$:  
\begin{equation}
\mathcal{N}_{\text{l}} = \{c^1_{1,1}, c^2_{2,2}, c^3_{3,3}, c^4_{4,4},c^5_{1,2},c^6_{2,3},c^7_{3,4}\}.
\end{equation}
The representation of $\mathcal{N}_{\text{l}}$ in Fig. \ref{fig:Example1} is referred to as ``geometric digraph'', where the circles are the nodes, the arrows are the arcs, and the geometry of each arc is specified. Specifically, $l_k$ is the length of the path travelled by the $k$-th compartment $c^k_{i,j}$. The configuration $\mathcal{N}_{\text{l}}$ is linear because: first, all the material is extracted from nonrenewable reservoirs; and second, all the material is sent to a landfill after its first use.  
\begin{figure*}
    \vspace{-0.3cm}
    \centering
    \includegraphics[width=0.85\textwidth]{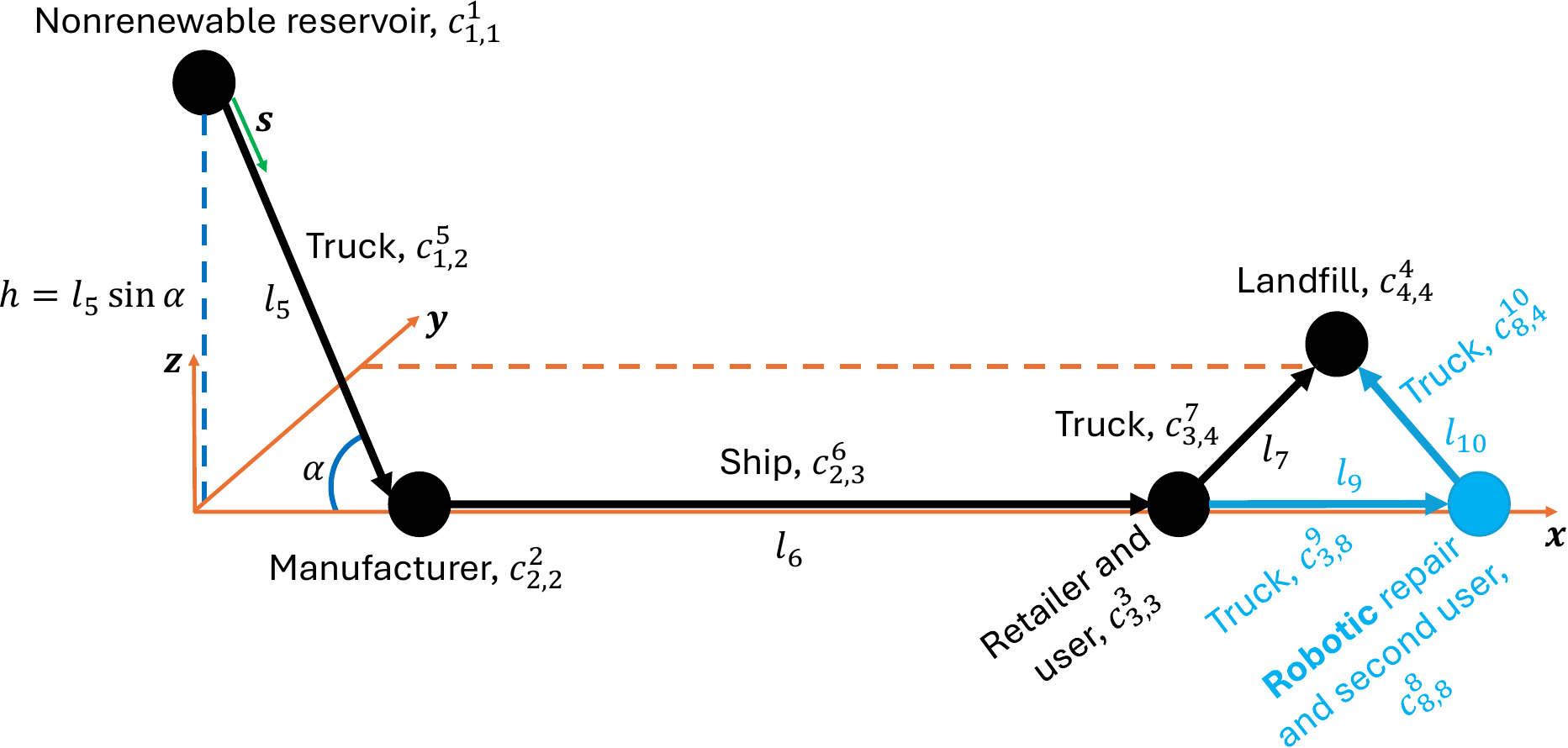}
    \caption{Geometric diagraph of Examples 1, 2, and 3.}
    \label{fig:Example1}
    \vspace{-0.5cm}
\end{figure*}  

The dynamics of the cube of mass $m$ and material $\beta$ can be described by considering one compartment at a time starting from the nonrenewable reservoir ($c^1_{1,1}$). Assume that the cube leaves the reservoir in $t_0 = 0$. Let $F_{\text{t}}$ be the force of the truck ($c^5_{1,2}$), let $F_{\text{f}}$ be the  %
rolling friction and overall motion resistance, let $m_\text{t}$ be the mass of the truck, and let $m_{\text{c}_5} = m_\text{t} + m$ be the total mass of the truck compartment during the transportation of the cube. Hence, the kinetic and the potential energies of the system are
\begin{equation}
T = \frac{1}{2} m_{\text{c}_5} \dot{s}^2(t), \quad 0 \leq s(t) \leq l_5,
\end{equation}
and
\begin{equation}
V = m_{\text{c}_5} g z(t) = m_{\text{c}_5} g (h - s(t) \sin \alpha), \quad 0 \leq s(t) \leq l_5,
\end{equation}
respectively. Thus, with $\xi = F_{\text{t}} - F_{\text{f}}$, the Lagrange's equation (\ref{eq:LagrangeEquation}) yields
\begin{equation}\label{eq:c5-ddots}
m_{\text{c}_5} \ddot{s}(t) = m_{\text{c}_5} g \sin \alpha + F_{\text{t}} - F_{\text{f}}, \quad 0 \leq s(t) \leq l_5,
\end{equation}
while the dynamics with respect to the reference frame is
\begin{equation}
\ddot{x}(t)  = \ddot{s}(t) \cos \alpha, \quad 0 \leq s(t) \leq l_5, 
\end{equation}
\begin{equation}
\ddot{y}(t)  = 0, \quad 0 \leq s(t) \leq l_5, 
\end{equation}
and
\begin{equation}\label{eq:c5-ddotz}
\ddot{z}(t)  = - \ddot{s}(t) \sin \alpha, \quad 0 \leq s(t) \leq l_5. 
\end{equation}
The positions and velocities of the material can be derived by integrating (\ref{eq:c5-ddots})-(\ref{eq:c5-ddotz}).
Assume now that the truck arrives at the manufacturer ($c^2_{2,2}$) at $t_1$ and that a manufactured product made of the cube of material $\beta$ leaves the manufacturer at $t_2$ to be shipped to the retailer and the user ($c^3_{3,3}$), where $t_2 > t_1$. This yields
\begin{equation}
\ddot{s}(t_1) = \dot{s}(t_1) = 0, s(t_1) = l_5, \quad t_1 < t < t_2.
\end{equation}
Then, the dynamics of the shipping compartment $c^6_{2,3}$ is
\begin{equation}
m_{\text{c}_6} \ddot{s}(t) = F_{\text{s}} - F_{\text{d}}, \quad l_5 \leq s(t) \leq l_5 + l_6,
\end{equation}
where $m_\text{s}$ is the mass of the empty ship, $m_{\text{c}_6} = m_\text{s} + m$ is the mass of the loaded ship, and where $F_{\text{s}}$ and $F_{\text{d}}$ are the force of the ship and the drag, respectively. The ship travels $l_6$, with $l_6 >> l_5$. The dynamics with respect to the reference frame for the shipping compartment $c^6_{2,3}$ yields
\begin{equation}
\ddot{x}(t) = \ddot{s}(t), \quad l_5 \leq s(t) \leq l_5 + l_6,  
\end{equation}
while $\ddot{y}(s(t)) = \ddot{z}(s(t)) = 0$ for $l_5 \leq s(t) \leq l_5 + l_6$.
The product made of the cube of material $\beta$ arrives at the retailer at $t_3$ and, after use, it is collected by a garbage truck ($c^7_{3,4}$) at $t_4$. Thus, the law of motion in $c^3_{3,3}$ is
\begin{equation}
\ddot{s}(t_3) = \dot{s}(t_3) = 0, s(t_3) = l_6, \quad t_3 < t < t_4.
\end{equation}
Then, the dynamics of the garbage truck compartment $c^7_{3,4}$ is
\begin{equation}
m_{\text{c}_5} \ddot{s}(t) = F_{\text{t}} - F_{\text{f}}, \quad l_5 + l_6 \leq s(t) \leq l_5 + l_6 + l_7,
\end{equation}
where we assumed a garbage truck with the same mass of the truck $c^5_{1,2}$. The garbage truck travels $l_7$, with $l_6 >> l_5 > l_7$. The dynamics of the material for $c^7_{3,4}$ with respect to the reference frame yields
\begin{equation}
\ddot{y}(t) = \ddot{s}(t), \quad l_5 + l_6 \leq s(t) \leq l_5 + l_6 + l_7,  
\end{equation}
while $\ddot{x}(t) = \ddot{z}(t) = 0$ for $l_5 + l_6 \leq s(t) \leq l_5 + l_6 + l_7$.
Finally, the garbage truck delivers the material to a landfill ($c^4_{4,4}$).

Let us now quantify the circularity (\ref{eq:circularity}) of $\mathcal{N}_{\text{l}}$, namely, $\lambda_{\text{l}}$. Since the material is only transported in batches, we have that $\dot{m}_{\text{u},\text{c}} = 0$; moreover, there is a batch of mass $m$ exiting a nonrenewable reservoir and a batch of mass $m$ entering a landfill; thus, from Definition \ref{def:unsMassFlow} it follows that $m_{\text{u},\text{b}} = 2m$, and hence,
\begin{equation}
\lambda_{\text{l}} = - 2m.
\end{equation}

\subsection{Example 2: Case of Reducing}
Four key practices in a circular economy are reduce, reuse, repair, and recycle \cite{potting2017circular}. Let us see how reducing the material demand in Example \ref{sub:Example1} affects the circularity, which in this case is indicated with $\lambda_{\text{r}_1}$. A possible reduction could result from the choice of the manufacturer ($c^2_{2,2}$) to use 50\% of renewable material to make its products, e.g., to use 50\% bio-plastic and 50\% synthetic plastic. In this case, a mass of $m/2$ exits the nonrenewable reservoir, while a mass of $m$ enters the landfill (the whole product); thus,  $m_{\text{u},\text{b}} = \frac{m}{2} + m$ and the circularity becomes
\begin{equation}
\lambda_{\text{r}_1} = - 1.5m.
\end{equation}
Note that $\lambda_{\text{r}_1} > \lambda_{\text{l}}$, thus the reduction improves circularity. Another case of reduction occurs by reducing the amount of material required to make a product. In this second case of reduction, the manufacturer may need only 80\% of the material $\beta$ required in Example \ref{sub:Example1}; this yields $m_{\text{u},\text{b}} = 0.8m + 0.8m = 1.6m$, thus the circularity with reduction becomes   
\begin{equation}
\lambda_{\text{r}_2} = - 1.6m.
\end{equation}
Note that $\lambda_{\text{r}_1} > \lambda_{\text{r}_2} > \lambda_{\text{l}}$.

\subsection{Example 3: Case of Robotic Repair}
Consider the arcs and the nodes in black and in blue in Fig. \ref{fig:Example1} with the exception of $c^7_{3,4}$. This yields a configuration involving robotic repair and second use ($c^8_{8,8}$), namely, $\mathcal{N}_{\text{p}}$:
\begin{equation}
\mathcal{N}_{\text{p}} = \{c^1_{1,1}, c^2_{2,2}, c^3_{3,3}, c^4_{4,4},c^8_{8,8},c^5_{1,2},c^6_{2,3},c^9_{3,8},c^{10}_{8,4}\}.
\end{equation}
The dynamics of $\beta$ in the case of $\mathcal{N}_{\text{p}}$ can be deduced from the modeling of $\mathcal{N}_{\text{l}}$ in Example \ref{sub:Example1}, and hence, it is omitted due to space constraints. Let us focus the attention on how the circularity is affected by the robotic repair and the second user, i.e., the node $c^8_{8,8}$ (refer to this recent work of the authors \cite{zocco2024unification} for further reading on ``circular robotics''). The circularity in the case of repair is indicated with $\lambda_{\text{p}}$. 

Equation (\ref{eq:circularity}) yields that the circularity with the repair stage is equivalent to the linear case in Example 1 because the mass of material $m$, eventually, reaches the landfill ($c^4_{4,4}$). Specifically, it gives that $\lambda_{\text{p}} = \lambda_{\text{l}} = -2m$. However, in Example 1, the material enters the landfill at $t_{5_\text{l}}$ after that the garbage truck has collected the used product in $c^3_{3,3}$ at $t_4 < t_{5_\text{l}}$. In contrast, in $\mathcal{N}_{\text{p}}$, at $t_4$, the material leaves $c^3_{3,3}$ to be repaired, and then, reused. Thus, it holds that $t_4 < t_{5_\text{l}} << t_{5_\text{p}}$, where $t_{5_\text{p}}$ is the time at which the material enters the landfill in the case of repair and reuse (i.e., in $\mathcal{N}_{\text{p}}$). This phenomenon shows that circularity has a \emph{memory}, that is, it must be defined with respect to a \emph{time horizon}, which we call $\phi$. To embed the memory into the notion of circularity (\ref{eq:circularity}), we use the notation $\lambda_{\phi,\text{l}}$ and $\lambda_{\phi,\text{p}}$ for the linear and the repair case, respectively; comparisons must be done for the same memory $\phi$ to be meaningful. 

Calculating the circularity with respect to $\phi = t_{5_\text{l}}$ yields $\lambda_{\phi,\text{l}}|_{\phi = t_{5_\text{l}}} = -2m$ and $\lambda_{\phi,\text{p}}|_{\phi = t_{5_\text{l}}} = -m$ for the linear and the repair case, respectively, which highlights the benefit of repair. In contrast, if $\phi = t_{5_\text{p}}$, we have that $\lambda_{\phi,\text{l}}|_{\phi = t_{5_\text{p}}} = \lambda_{\phi,\text{p}}|_{\phi = t_{5_\text{p}}} = -2m$, meaning that the linear and the repair cases have the same circularity for $t > t_{5_\text{p}}$. With the configuration $\mathcal{N}_{\text{p}}$, the mass of material $m$ is kept in use for a longer time, and specifically, for the time $\Delta_{\text{e}} = t_{5_\text{p}} - t_{5_\text{l}}$. Extending the life cycle of materials is a key principle of CE \cite{EllenMacArthurFound} and it is captured by $\lambda(\mathcal{N})$ with memory, namely, $\lambda_{\phi}(\mathcal{N})$.

\section{Conclusion}
This paper proposes a physics-based, quantitative approach for assessing the circularity of a process or a system. In particular, a quantitative measure of circularity is proposed and its evaluation in three representative examples is shown. The proposed preliminary definition can be used to assess the circularity of a system and optimize its structure to minimize unsustainable mass transfers. Further studies on this topic aim to provide a non-dimensional, weighted definition of circularity and unsustainability, suitable for performing direct comparisons between different systems and plants, that, besides mass transfer, considers material typology in terms of availability and risks. Thanks to the physics-based process model, other parameters related to system dynamics can be evaluated and integrated into the circularity evaluation. Future investigations on this topic will focus on the integration of circularity assessment to take into account energetic efficiency, time constraints, and environmental impact.

\bibliographystyle{splncs04}
\bibliography{references}

\end{document}